\documentclass[twoside, 12pt]{amsart}
\usepackage{latexsym}
\usepackage{amssymb,amsmath,amsopn}
\usepackage[dvips]{graphicx}   %To insert ps figures
\usepackage{color,epsfig}      %To insert ps figures
\usepackage{url}

\makeatletter
\def\@settitle{\begin{center}%
  \baselineskip14\p@\relax
  \bfseries
  \uppercasenonmath\@title
  \@title
  \ifx\@subtitle\@empty\else
     \\[1ex]\uppercasenonmath\@subtitle
     \footnotesize\mdseries\@subtitle
  \fi
  \end{center}%
}
\def\subtitle#1{\gdef\@subtitle{#1}}
\def\@subtitle{}
\makeatother

\newtheorem{rem}{Remark}
\newtheorem{cor}{Corollary}

\newtheorem{thm}{Theorem}

\newtheorem{lem}{Lemma}

\newtheorem{conj}{Conjecture}

\renewcommand{\Re}{\mathbb{R}}

\newcommand{\Eu}{\mathbb{R}}

\begin{document}

%------
% Insert the title of your paper and (if necessary)
% a short title for the running head.
%------
\title{Covering the Crosspolytope with Crosspolytopes}
%\titlemark{Hadwiger's covering problem for the Crosspolytope}

%------
% Insert full names of the authors.
% Add further authors as follows:
%  \emsauthor{2}{}{}
%  \emsauthor{3}{}{}
% etc.
% Abbreviate first names for the running head.
%------
\author{Antal Jo\'os}
\address{Antal Jo\'os\\ Department of Mathematics, University of Duna\'ujv\'aros, T\'ancsics M. u. 1/a, Duna\'ujv\'aros, Hungary, 2400}
\email{joosa@uniduna.hu}


\begin{abstract}
Let $\gamma^d_m(K)$ be the smallest positive number $\lambda$ such that the convex body $K$ can be
covered by $m$ translates of $\lambda K$. Let $K^d$ be the $d$-dimensional crosspolytope. It will be proved that $\gamma^d_m(K^d)=1$ for $1\le  m< 2d$, $d\ge4$; $\gamma^d_m(K^d)=\frac{d-1}{d}$ for $m=2d,2d+1,2d+2$, $d\ge4$; $\gamma^d_m(K^d)=\frac{d-1}{d}$ for $ m= 2d+3$, $d=4,5$;
$\gamma^d_m(K^d)=\frac{2d-3}{2d-1}$ for $ m= 2d+4$, $d=4$ and
$\gamma^d_m(K^d)\le\frac{2d-3}{2d-1}$ for $ m= 2d+4$, $d\ge5$. Moreover the Hadwiger’s covering conjecture is verified for the $d$-dimensional crosspolytope.
\end{abstract}
\maketitle

%------
% INSERT THE BODY OF THE PAPER HERE (except
% acknowledgments, funding info and bibliography)
%------
\section{Introduction}\label{sec:intro}

Let $\Eu^d$ be the $d$-dimensional Euclidean space and let $\mathcal{K}^d$ be the set of all convex bodies in $\Eu^d$ with non-empty interior.
Let $p$ and $q$ be points in $\Eu^d$. Let $[p, q]$, $|pq|$ and
$\overrightarrow{pq}$ denote, respectively, the line segment, the distance and the vector with initial point $p$ and terminal point $q$.
If $K\in \mathcal{K}^d$, then let $c^d(K)$ denote
the \emph{covering number} of $K$, i.e., the smallest number of translates of the interior of $K$ such that their union can cover $K$.
Levi \cite{Levi} proved in 1955 that
$$c^2(K)=
\left\{
\begin{array}{cc}
  4 & \text{if } K \text{ is a parallelogram,} \\
  3 & \text{otherwise.} \\
\end{array}\right.$$
In 1957 Hadwiger \cite{Hadwiger} considered this question in any dimensions and posed
\begin{conj}[Hadwiger's covering problem]\label{conj:Hadwiger}
  For every $K \in  \mathcal{K}^d$ we have
$c^d(K) \le 2^d,$
where the equality holds if and only if $K$ is a parallelepiped.
\end{conj}
In the literature the Boltyanski’s illumination problem is a similar problem (see. e.g. \cite{Boltyanski-Martini-Soltan}, \cite{Wu},
\cite{Boltyanski-Martini}).
Lassak \cite{Lassak84} proved in 1984 the Hadviger conjecture for the three dimensional centrally symmetric bodies. Rogers and Zong \cite{RogersZong}
presented the upper bound
$c^d(K)\le{2d\choose d}(d\log d+d\log \log d+5d)$ for general $d$-dimensional convex bodies and $c^d(K)\le{2^d}(d\log d+d\log \log d+5d)$ for centrally symmetric  $d$-dimensional convex bodies.
\newpage \noindent Using the idea of \cite{Artstein-Avidan} Huang et. al. \cite{Huang} presented the upper bound $c^d(K)\le c_1 4^de^{-c_2\sqrt{d}}$ for some universal constants $c_1,c_2>0$. K. Bezdek \cite{Bezdek92} proved the Conjecture \ref{conj:Hadwiger}
for convex  polyhedron  in $\Eu^3$ having anaffine symmetry and 
Dekster \cite{Dekster} verified it for convex bodies in $\Eu^3$ symmetric about a plane.

For $K\in \mathcal{K}^d$ and any positive integer $m$, let $\gamma^d_m(K)$ be the smallest positive number $\lambda$ such that $K$ can be
covered by $m$ translations of $\lambda K$, i.e.,
$$\gamma^d_m(K) = \min\left\{\lambda \in\Re^+ :\exists (u_1,\ldots,u_m)\in(\Eu^d)^m, \text{ s.t. } K \subseteq \bigcup\limits_{i=1}^{m}(\lambda K + u_i)\right\}.$$
This is called the \emph{covering functional} of $K$ with respect to $m$ \cite{Wu-Zhang-He} or the \emph{$m$-covering number} of $K$ \cite{Lassak86} \cite{Zong10}.
A translate of $\lambda K^d$ is called a \emph{homothetic copy} of $K^d$. In the following we use the short notation $\lambda K^d$ for a homothetic copy of $K^d$ instead of $\lambda K^d+{\bf v}$ where $\bf v$ is a vector.
Observe,  $\gamma^d_m(K) = 1$, for all $m \le d$, and $\gamma^d_m(K)$ is a non-increasing step sequence for all positive integers $m$
and all convex bodies $K$.
Now $c^d(K) \le m$ for some $m \in Z^+$ if and only if $\gamma ^d_m(K) < 1$ \cite{Boltyanski-Martini-Soltan}. 
Estimating covering functionals of convex bodies plays a crucial role in Chuanming Zong’s quantitative program for attacking Conjecture 1 (see \cite{Zong10} for more details).
Lassak \cite{Lassak86} showed that for every two-dimensional convex domain K,
$\gamma^2_4(K) \le
{\sqrt{2}\over2}.$
Zong \cite{Zong} proved
$\gamma^3_8(C) \le
{2\over3}$
for a bounded three-dimensional convex cone $C$, and
$\gamma^3_8(B_p) \le
\sqrt{2\over3}$
for all the unit ball $B_p$ of the three-dimensional $l_p$ spaces. Wu and He \cite{Wu-He} estimated the value of $\gamma^3_m(P)$ where $P$ is a convex polytope.
Wu et al. \cite{Wu-Zhang-He} determined the value of $\gamma^3_m(K)$ where $K$ is the union of two compact convex sets having no interior points.
(See   \cite{Bezdek92},  \cite{Bezdek12}, \cite{Bezdek-Bisztriczky97}, \cite{Bezdek-Langi-Naszodi}, \cite{Bezdek-Muhammad}, \cite{Bisztriczky-Fodor}, \cite{Boltyanski-Martini-Soltan}, 
\cite{Boltyanski2001}, \cite{Dekster}, \cite{Hadwiger}, \cite{He-Lv-Martini-Wu}, \cite{Ivanov-Strachan},   \cite{Livshyts-Tikhomirov}, \cite{Naszodi}, \cite{Papadoperakis},  \cite{Prymak-Shepelska},   \cite{Schramm}, \cite{Talata} for more information.)

Let $K^d$ be the cross polytope in the $d$-dimensional Euclidean space with diameter $2$, that is, $K^d=\{(x_1,\ldots,x_d):|x_1|+\ldots+|x_d|\le 1\}$.
In 2021 Lian and Zhang \cite{LianZhang}  proved
$$\gamma^3_m(K^3)=
\left\{
\begin{array}{rcl}
  1 & \text{if }& m=1,\ldots,5, \\
  2/3 & \text{if }& m=6,\ldots,9, \\
  3/5 & \text{if }& m=10,\ldots,13, \\
  4/7 & \text{if }& m=14,\ldots,17. \\
\end{array}\right.$$
Let $K\subset \Eu^d$ be a convex body, and denote by $r$ and $s$ points in $K$ such that
$\overrightarrow{rs}\parallel\overrightarrow{pq}$ and $|rs|\ge|r's'|$ where $\{r', s'\}\subset K$ and
$r's'\parallel pq$. The $K$-length of $[p, q]$, or
equivalently, the $K$-distance of $p$ and $q$ is $2|pq|/|rs|$, and it is  denoted by $d_K(p, q)$. If $K$ is the Euclidean $d$-ball, then $d_K(p, q)$ is the Euclidean distance. Let $||p||_s$ be the $s$-norm of $p$, i.e. if $p=(p_1,\ldots,p_d)$, then $||p||_s=\sqrt[s]{\sum_{i=1}^{d}|p_i|^s}$ for $s\ge 1$. The $2$-norm of $p-q$ is the Euclidean distance of the points $p$ and $q$.
Observe, the $d_{K^d}(p,q)$ is the distance of $p-q$ in $1$-norm, i.e. $d_{K^d}(p,q)=||p-q||_1$.
%Let $o$ be the origin of the Cartesian coordinate system in the $d$ dimensional Euclidean space.
%For two sets $A$ and $B$, $d(A,B)$ denotes the (Euclidean) distance between the sets.

\begin{rem}
Let $p$ and $q$ be different points in $\Eu^d$. If $[p,q]$ lies in a homothetic copy $\lambda K^d$, then $\lambda \ge {1\over2}d_{K^d}(p,q)={1\over2}||p-q||_1$.
\end{rem}

\begin{rem}
Let $p$ be a point in $\Eu^d$. If $p$ lies in a homothetic copy $\lambda K^d$ and $c$ is the centre of this homothetic copy $\lambda K^d$, then $d_{K^d}(c,p)=||c-p||_1 \le \lambda$.
\end{rem}

\begin{thm}\label{thm:<2d}
  If $1\le m<2d$, then $\gamma_m^d(K^d)=1$.
\end{thm}

\begin{proof}
Of course, if $\lambda=1$, then there is a homothetic copy $\lambda K^d$ such that $K^d$ is covered by $\lambda K^d$. 
Thus $\gamma_m^d(K^d)\le1$. Assume $\gamma_m^d(K^d)=\mu<1$.
Since the number of the vertices of $K^d$ is $2d$, then there is a homothetic copy $\mu K^d$ which contains two vertices of $K^d$. Let $v_1$ and $v_2$ be these two vertices.
Observe, the opposite facets of $K^d$ are parallel and any edge of $K^d$ connects two vertices of opposite facets. By Remark 1, $\mu\ge{1\over2}||v_1-v_2||_1={1\over2}d_{K^d}(v_1,v_2)=1$, a contradiction.
\end{proof}

\begin{lem}\label{lem:center is not covered}
  Let $F$ be a facet of $K^d$. If a vertex $v$ of $F$ is covered by $K^d_1$ a  homothetic copy of $K^d$ with ratio $\lambda$ ($0<\lambda<1$), then $F\cap K^d_1$ is contained in the homothetic image of $F$ with ratio $\lambda$ and centre $v$.
\end{lem}

\begin{proof}
Let $\alpha$ be the hyperplane $x_1+\ldots +x_d=1$.
Without loss of generality it may be assumed that $v=(0,\ldots,0,1)$ and $F$ is the facet on the hyperplane $\alpha$. Let $\mathcal A$ be the affine transformation such that $K^d_1=\mathcal{A}(K^d)$. It may be assumed that $\mathcal{A}(o)$ lies in $K^d$. Let $w(w_1,\ldots,w_d)=\mathcal{A}(v)$.
Let us assume that $o\not \in K^d_1$ (the opposite case is similar).
Let $F_1$ be the facet of $K^d$ on the hyperplane $x_1+\ldots +x_{d-1}-x_d=1$.
Let $p(p_1,\ldots,p_d)$ be a point of $F\cap \mathcal{A}(F_1)$.
First we will see that $p_d\ge 1-\lambda$.
Observe, the equation of the hyperplane containing $\mathcal{A}(F_1)$ is
$$x_1-w_1+x_2-w_2+\ldots+x_{d-1}-w_{d-1}-x_d+w_d-\lambda=\lambda.$$
Since $p$ lies on $F$ and $\mathcal{A}(F_1)$, we have
\begin{equation}\label{eq:F}
p_1+\ldots+p_d=1
\end{equation}
and
\begin{equation}\label{eq:AF1}
p_1+p_2+\ldots+p_{d-1}-p_d=w_1+\ldots+w_{d-1}-w_d+2\lambda.
\end{equation}
From (\ref{eq:F})-(\ref{eq:AF1}) we have
\begin{equation}\label{eq:3}
p_d=\frac{1}{2}+\frac{1}{2}(-w_1-\ldots-w_{d-1}+w_{d})-\lambda
\end{equation}
Observe, the point $w$ lies in the translated image of $K^d$ by the vector $[0,\ldots,0,2]^T$.
Since $v$ is covered by $K^d_1$, the point $w$ lies in the halfspace bounded by the hyperplane $-x_1-\ldots-x_{d-1}+x_d=1$ and does not contain the origin. Thus $$-w_1-\ldots-w_{d-1}+w_{d}\ge 1.$$
Substituting this into (\ref{eq:3}), we have
$$p_d\ge \frac{1}{2}+\frac{1}{2}-\lambda=1-\lambda$$
Observe $K^d_1$ lies in the halfspace bounded by the hyperplane $\mathcal{A}(F_1)$ and containing the vertex $v$. Moreover $p$ is an arbitrary point of the intersection of this halfspace and $F$. This means any point from the intersection has the property that the last coordinate of the point is not less than $1-\lambda$, which completes the proof of the lemma.
\end{proof}

\begin{thm}\label{thm:2d}
  If $m=2d$, then $\gamma_m^d(K^d)={d-1\over d}$.
\end{thm}

\begin{proof}
Let $\Gamma=\{{d-1\over d}K^d+(\pm{1\over d},0\ldots,0),\ldots,{d-1\over d}K^d+(0\ldots,0,\pm{1\over d})\}$.
Since the homothetic copies in $\Gamma$ cover the surface of $K^d$ and the origin lies in each element in $\Gamma$, $K^d$ is covered by $\Gamma$. Thus $\gamma_m^d(K^d)\le{d-1\over d}$.\\
Assume $\gamma_m^d(K^d)=\mu<{d-1\over d}$.
Since $\mu<{d-1\over d}<1$, there is no  homothetic copy containing two vertices of $K^d$.  Let $F$ be a facet of $K^d$ lying on the hyperplane $x_1+\ldots+x_d=1$ and let $c$ be the centre of the $(d-1)$-simplex $F$. Observe $c(1/d,\ldots,1/d)$. By Lemma 1, $c$ and a vertex of $F$ cannot lie in the same homothetic copy $\mu K^d$. Since ${d-1\over d}<1$, $c$ and a vertex of $K^d$ on the plane $x_1+\ldots+x_d=-1$ cannot lie in the same homothetic copy $\mu K^d$, a contradiction.
\end{proof}

Observe, the proof of Theorem 2 comes from Remark 1. Indeed, the $1$-norm of $||c-v||_1\ge 2\frac{d-1}{d}$ where $v$ is a vertex of $K_1^d$.

\begin{cor}
  We have $c^d(K^d)\le2d$ for $d\ge4$.
	The Hadwiger’s covering problem is solved for the $d$-dimensional crosspolitope for $d\ge4$.
\end{cor}

\begin{thm}\label{thm:2d+2}
  If $m=2d+1,2d+2$, then $\gamma_m^d(K^d)={d-1\over d}$ for $d\ge 4$.
\end{thm}

\begin{proof}
Since $$\gamma_{2d+2}^d(K^d)\le\gamma_{2d+1}^d(K^d)\le \gamma_{2d}^d(K^d)={d-1\over d},$$
from $\gamma_{2d+2}^d(K^d)={d-1\over d}$ it comes the statement of the theorem.
Let it be assumed that $\gamma_{2d+2}^d(K^d)=\mu<{d-1\over d}$.
Let $c_1,\ldots,c_{2^d}$ be the centres of the facets of $K^d$ and let $C$ be the $d$-cube with vertices $c_1,\ldots,c_{d^2}$. By Lemma \ref{lem:center is not covered},  $2d$ homothetic copies of $K^d$ can cover the $2d$ vertices of $K^d$ and the points $c_1,\ldots,c_{2^d}$ are uncovered by these homothetic copies.
Let $S$ be the set containing the vertices of $K^d$ and the centres $(1/d,\ldots,1/d)$ and $(-1/d,\ldots,-1/d)$. By Remark 1, any two points from $S$ cannot lie in a homothetic copy of $K^d$ with ratio $\mu$, a contradiction.
\end{proof}

\begin{lem}\label{lem:diagonal}
  Let $c_1,\ldots,c_{2^d}$ be the centres of the facets of $K^d$. If $\lambda<1$ and the facets containing $c_i$ and $c_j$ are parallel facets, then a homothetic copy $\lambda K^d$ cannot contain $c_i$ and $c_j$.
\end{lem}

\begin{proof}
Without loss of generality it may be assumed that the two hyperplanes are $\alpha_1:x_1+\ldots+x_d=1$ and $\alpha_2:x_1+\ldots+x_d=-1$. In this case the centres are $c_i(-1/d,\ldots, -1/d)$ and $c_j(1/d,\ldots, 1/d)$.
%The vector $\overrightarrow{c_ic_j}=[2/d,2/d,\ldots,2/d]^T$ is perpendicular to the hyperplanes $\alpha_1$ and $\alpha_2$. Thus the distance between the hyperplanes $\alpha_1$ and $\alpha_2$ is the length of the longest chord in $K^d$ parallel to $c_ic_j$. Observe, $d(c_i,c_j)=d(\alpha_1,\alpha_2)$ and
Now $||c_i-c_j||_1=d_{K^d}(c_i,c_j)=2$. By Remark 1, a homothetic copy $\lambda K^d$ cannot contain $c_i$ and $c_i$.
\end{proof}
\begin{lem}\label{lem:face diagonal}
  Let $c_1,\ldots,c_{2^d}$ be the centres of the facets of $K^d$. If $\lambda<{d-1\over d}$ and the facets containing $c_i$ and $c_j$ have only a vertex in common, then a homothetic copy $\lambda K^d$ cannot contain $c_i$ and $c_j$.
\end{lem}

\begin{proof}
Without loss of generality it may be assumed that $c_i(1/d,-1/d,\ldots,-1/d)$ and  $c_j(1/d,\ldots,1/d)$. Now the common vertex of the facets of $c_i$ and $c_j$ is $(1,0,\ldots,0)$.
%Thus $c_ic_j={2\over d}\sqrt{d-1}$.
%Let $F$ ($F'$, resp.) be the $(d-2)$-flat containing the vertices
%$v_1(0,1,0,\ldots,0)$, $v_2(0,0,1,0,\ldots,0)$, $\ldots$, %$v_{d-1}(0,0,0,\ldots,0,1)$ ($v_1'(0,-1,0,\ldots,0)$, %$v_2'(0,0,-1,0,\ldots,0)$, $\ldots$, $v_{d-1}'(0,0,0,\ldots,0,-1)$, %resp.).
%The vector $\overrightarrow{c_ic_j}=[0,2/d,\ldots,2/d]^T$ is not %perpendicular to any facets of $K^d$, but it is perpendicular to the %$(d-2)$-flats $F$ and $F'$. Thus the distance between the $(d-2)$-flats %$F$ and $F'$ is the length of the longest chord in $K^d$ parallel to %$c_ic_j$. Since
%$$d(F,F')=d\left((0,{1\over d-1},\ldots,{1\over d-1}),(0,{-1\over %d-1},\ldots,{-1\over d-1})\right)={2\over \sqrt{d-1}},$$
%$${|\overrightarrow{c_ic_j}|\over d(F,F')}={d-1\over d}$$
Thus $||c_i-c_j||_1=d_{K^d}(c_i,c_j)=2{d-1\over d}$ and by Remark 1, a homothetic copy $\lambda K^d$ cannot contain $c_i$ and $c_i$.
\end{proof}

\begin{thm}
 If $m=2d+3$, then $\gamma_m^d(K^d)={d-1\over d}$ for $d= 4,5$.
\end{thm}

\begin{proof}
By Theorem \ref{thm:2d+2}, $\gamma_{2d+3}^d(K^d)\le
{d-1\over d}$. Let it be assumed that $\gamma_{2d+3}^d(K^d)=\mu<{d-1\over d}$.
Let $c_1,\ldots,c_{2^d}$ be the centres of the facets of $K^d$ and let $C$ be the $d$-cube with vertices $c_1,\ldots,c_{2^d}$. By Lemma \ref{lem:center is not covered}, the $2d$ homothetic copies of $K^d$ can cover the $2d$ vertices of $K^d$ and the points $c_1,\ldots,c_{2^d}$ are uncovered by these homothetic copies.
It will be proved that the vertices of the $d$-cube $C$ cannot be covered by $3$ homothetic copies of $K^d$ with ratio $\mu$.
Let us assume, that the vertices of the $d$-cube $C$ can be covered by $3$ homothetic copies of $K^d$ with ratio $\mu$.
Observe, $c_i$ and $c_j$ lie on two parallel facets of $K_{d}$ if and only it $c_i$ and $c_j$ are the endpoints of a diagonal of $C$. Thus
if $c_i$ and $c_j$ are the endpoints of a diagonal of the $d$-cube $C$, then by Lemma \ref{lem:diagonal}, $c_i$ and $c_j$ cannot be covered a homothetic copy $\mu K^d$.
Observe, $c_i$ and $c_j$ lie on two facets and the two facets have only one vertex in common if and only if $c_i$ and $c_j$ are the endpoints of a diagonal of a facet of $C$.
Thus
if $c_i$ and $c_j$ are the endpoints of a diagonal of a facet of $C$, then by Lemma \ref{lem:face diagonal}, $c_i$ and $c_j$ cannot be covered by a homothetic copy $\mu K^d$.
We distinguish 2 cases.

{\bf Case 1.} We have $d=4$.\\
By the Pigeonhole principle there is a homothetic copy - say $H_1$ - that $H_1$ contains $6$ vertices of the $4$-cube $C$.
Without loss of generality it may be assumed that $c_1=(1/4,1/4,1/4,1/4)$ lies in $H_1$. By Lemma \ref{lem:diagonal}, the vertex $(-1/4,-1/4,-1/4,-1/4)$ cannot lie in $H_1$.
By Lemma \ref{lem:face diagonal}, the vertices $(1/4,-1/4,-1/4,-1/4)$, $(-1/4,1/4,-1/4,-1/4)$, $(-1/4,-1/4,1/4,-1/4)$ and $(-1/4,-1/4,-1/4,1/4)$ cannot lie in $H_1$. Let $c_2=(-1/4,\linebreak1/4,1/4,1/4)$, $c_3=(1/4,-1/4,1/4,1/4)$, $c_4=(1/4,1/4,-1/4,1/4)$, $c_5=(1/4,1/4,1/4,\linebreak-1/4)$, $c_6=(-1/4,-1/4,/4,1/4)$, $c_7=(-1/4,1/4,-1/4,1/4)$, $c_8=(-1/4,1/4,1/4,\linebreak-1/4)$, $c_9=(1/4,-1/4,-1/4,/4)$, $c_{10}=(1/4,-1/4,1/4,-1/4)$, $c_{11}=(1/4,1/4,-1/4,\linebreak-1/4)$ and $S^4=\{c_2,c_3,c_4,c_5\}$.
We distinguish 5 subcases.

{\bf Subcase 1.1.} The homothetic copy $H_1$ contains the vertices $c_2$, $c_3$, $c_4$ and $c_5$.\\
By Lemma \ref{lem:face diagonal}, $H_1$ cannot contain the vertices $c_6, \ldots,c_{11}$. Thus $H_1$ cannot contain 6 vertices of $C$, a contradiction.

{\bf Subcase 1.2.} The homothetic copy $H_1$ contains exactly three vertices from $S^4$.\\
Without loss of generality it may be assumed that $H_1$ contains $c_2$, $c_3$ and $c_4$.
By Lemma \ref{lem:face diagonal}, $H_1$ cannot contain the vertices $c_6, \ldots,c_{11}$. Thus $H_1$ cannot contain 6 vertices of $C$, a contradiction.

{\bf Subcase 1.3.} The homothetic copy $H_1$ contains exactly two vertices from $S^4$.\\
Without loss of generality it may be assumed that $H_1$ contains $c_2$ and $c_3$.
By Lemma \ref{lem:face diagonal}, $H_1$ cannot contain the vertices  $c_7, \ldots,c_{11}$.
Thus $H_1$ cannot contain 6 vertices of $C$, a contradiction.

{\bf Subcase 1.4.} The homothetic copy $H_1$ contains exactly one vertex from $S^4$.\\
Without loss of generality it may be assumed that $H_1$ contains $c_2$.
By Lemma \ref{lem:face diagonal}, $H_1$ cannot contain the vertices $c_9, \ldots,c_{11}$.
Thus $H_1$ cannot contain 6 vertices of $C$, a contradiction.

{\bf Subcase 1.5.} The homothetic copy $H_1$ does not contain any vertex from $S^4$.\\
By Lemma \ref{lem:diagonal}, $H_1$ cannot contain $c_{6}$ and $c_{11}$ in the same time.
By Lemma \ref{lem:diagonal}, $H_1$ cannot contain $c_{7}$ and $c_{10}$ in the same time.
Thus $H_1$ cannot contain 6 vertices of $C$, a contradiction.

{\bf Case 2.} We have $d=5$.\\
By the Pigeonhole principle there is a homothetic copy - say $H_1$ - that $H_1$ contains $11$ vertices of the $5$-cube $C$.
Without loss of generality it may be assumed that $c_1=(1/5,1/5,1/5,\\1/5,1/5)$ lies in $H_1$. By Lemma \ref{lem:diagonal}, the vertex $(-1/5,-1/5,-1/5,-1/5,-1/5)$ cannot lie in $H_1$.
By Lemma \ref{lem:face diagonal}, the vertices $(1/5,-1/5,-1/5,-1/5,-1/5)$,
$(-1/5,1/5,-1/5,-1/5,\linebreak-1/5)$,
$(-1/5,-1/5,1/5,-1/5,-1/5)$,
$(-1/5,-1/5,-1/5,1/5,-1/5)$ and
$(-1/5,-1/5,\linebreak-1/5,-1/5,1/5)$ cannot lie in $H_1$. Let $c_2=(-1/5,1/5,1/5,1/5,1/5)$, $c_3=(1/5,-1/5,1/5,\linebreak1/5,1/5)$, $c_4=(1/5,1/5,-1/5,1/5,1/5)$, $c_5=(1/5,1/5,1/5,-1/5,1/5)$, $c_6=(1/5,1/5,\linebreak1/5,1/5,-1/5)$,
$c_7=(-1/5,-1/5,1/5,1/5,1/5)$,
 $c_8=(-1/5,1/5,-1/5,1/5,1/5)$,
 $c_9=(-1/5,1/5,1/5,-1/5,1/5)$,
 $c_{10}=(-1/5,1/5,1/5,1/5,-1/5)$,
$c_{11}=(1/5,-1/5,-1/5,\linebreak 1/5,1/5)$,
$c_{12}=(1/5,-1/5,1/5,-1/5,1/5)$,
$c_{13}=(1/5,-1/5,1/5,1/5,-1/5)$,
$c_{14}=(1/5,1/5,-1/5,-1/5,1/5)$,
$c_{15}=(1/5,1/5,-1/5,1/5,-1/5)$,
$c_{16}=(1/5,1/5,1/5,-1/5,\linebreak-1/5)$,
$c_{17}=(-1/5,-1/5,-1/5,1/5,1/5)$,
$c_{18}=(-1/5,-1/5,1/5,-1/5,1/5)$,
$c_{19}=(-1/5,-1/5,1/5,1/5,-1/5)$,
$c_{20}=(-1/5,1/5,-1/5,-1/5,1/5)$,
$c_{21}=(-1/5,1/5,-1/5,\linebreak 1/5,-1/5)$,
$c_{22}=(-1/5,1/5,1/5,-1/5,-1/5)$,
$c_{23}=(1/5,-1/5,-1/5,-1/5,1/5)$,
$c_{24}=(1/5,-1/5,-1/5,1/5,-1/5)$,
$c_{25}=(1/5,-1/5,1/5,-1/5,-1/5)$,
$c_{26}=(1/5,1/5,-1/5,\linebreak-1/5,-1/5)$ and $S^5=\{c_2,c_3,c_4,c_5,c_6\}$.
We distinguish 6 subcases.

{\bf Subcase 2.1.} The homothetic copy $H_1$ contains the vertices $c_2$, $c_3$, $c_4$, $c_5$ and $c_6$.\\
By Lemma \ref{lem:face diagonal}, $c_{17},\ldots,c_{26}$ cannot lie in $H_1$.
Let it be assumed that $H_1$ contains $c_7$. (The cases $H_1$ contains $c_8,\ldots,c_{15}$, or $c_{16}$ are similar.) By Lemma \ref{lem:face diagonal}, $H_1$ cannot contain  $c_{14}$, $c_{15}$ or $c_{16}$.
By Lemma \ref{lem:face diagonal}, $c_{8}$ and $c_{12}$ cannot lie in $H_1$ in the same time. By Lemma \ref{lem:face diagonal}, $c_{9}$ and $c_{13}$ cannot lie in $H_1$ in the same time. By Lemma \ref{lem:face diagonal}, $c_{10}$ and $c_{11}$ cannot lie in $H_1$ in the same time.
Thus $H_1$ cannot contain 11 vertices of $C$, a contradiction.

{\bf Subcase 2.2.} The homothetic copy $H_1$ contains exactly four vertices from $S^5$.\\
Without loss of generality it may be assumed that $H_1$ contains $c_2$, $c_3$, $c_4$ and $c_5$.
By Lemma \ref{lem:face diagonal}, $H_1$ cannot contain the vertices $c_{17}, \ldots,c_{26}$.\\
Let us assume that $H_1$ contains $c_7$. (The cases that $H_1$ contains $c_8$, $c_9$, $c_{11}$, $c_{12}$ or $c_{14}$ are similar.)
By Lemma \ref{lem:face diagonal}, $H_1$ cannot contain $c_{14}$, $c_{15}$ or $c_{16}$. By Lemma \ref{lem:face diagonal}, $H_1$ cannot contain $c_{8}$ and $c_{12}$ in the same time. By Lemma \ref{lem:face diagonal}, $H_1$ cannot contain $c_{9}$ and $c_{11}$ in the same time.
Thus $H_1$ cannot contain 11 vertices of $C$, a contradiction.\\
Let us assume that $H_1$ contains $c_{10}$ (and does not contain the vertices $c_7,c_8,c_9,c_{11},c_{12}$ or $c_{14}$). (The cases that $H_1$ contains $c_{13}$, $c_{15}$ or $c_{16}$ are similar.)
Now, $H_1$ can contain the vertices $c_1,\ldots,c_5$, $c_{10}$, $c_{13}$, $c_{15}$,  $c_{16}$,
thus $H_1$ cannot contain 11 vertices of $C$, a contradiction.

{\bf Subcase 2.3.} The homothetic copy $H_1$ contains exactly three vertices from $S^5$.\\
Without loss of generality it may be assumed that $H_1$ contains $c_2$, $c_3$ and $c_4$.
By Lemma \ref{lem:face diagonal}, $H_1$ cannot contain the vertices $c_{18}, \ldots,c_{26}$.\\
Let us assume that $H_1$ contains $c_7$. (The cases $H_1$ contains $c_8$ or $c_{11}$ are similar.)
By Lemma \ref{lem:face diagonal}, $H_1$ cannot contain $c_{14}$, $c_{15}$ or $c_{16}$.
By Lemma \ref{lem:face diagonal}, $H_1$ cannot contain $c_{8}$ and $c_{12}$ in the same time. By Lemma \ref{lem:face diagonal}, $H_1$ cannot contain $c_{9}$ and $c_{13}$ in the same time. By Lemma \ref{lem:face diagonal}, $H_1$ cannot contain $c_{10}$ and $c_{11}$ in the same time.
Thus $H_1$ cannot contain 11 vertices of $C$, a contradiction.\\
Let us assume that $H_1$ contains $c_9$ (and does not contain the vertices $c_7,c_8$ or $c_{11}$). (The cases $H_1$ contains $c_{10}$, $c_{12}$, $c_{13}$, $c_{14}$ or $c_{15}$ are similar.)
By Lemma \ref{lem:face diagonal}, $H_1$ cannot contain $c_{11}$, $c_{13}$ or $c_{15}$.
Since $H_1$ can contain the $c_1,\ldots, c_4$, $c_9$, $c_{10}$, $c_{12}$,  $c_{14}$, $c_{16}$, $c_{17}$,
$H_1$ cannot contain 11 vertices of $C$, a contradiction.\\
Let us assume that $H_1$ contains $c_{16}$ (and does not contain the vertices $c_7,\ldots,c_{14}$ or $c_{15}$).
Since $H_1$ can contain the $c_1,\ldots, c_4$, $c_{16}$, $c_{17}$,
$H_1$ cannot contain 11 vertices of $C$, a contradiction.\\
Let us assume that $H_1$ contains $c_{17}$ (and does not contain the vertices $c_7,\ldots,c_{15}$ or $c_{16}$).
Since $H_1$ can contain the $c_1,\ldots, c_4$, $c_{17}$,
$H_1$ cannot contain 11 vertices of $C$, a contradiction.

{\bf Subcase 2.4.} The homothetic copy $H_1$ contains exactly two vertices from $S^5$.\\
Without loss of generality it may be assumed that $H_1$ contains $c_2$ and $c_3$.
By Lemma \ref{lem:face diagonal}, $H_1$ cannot contain the vertices $c_{20}, \ldots,c_{26}$.\\
Let us assume that $H_1$ contains $c_7$.
By Lemma \ref{lem:face diagonal}, $H_1$ cannot contain $c_{14}$, $c_{15}$ or $c_{16}$.
By Lemma \ref{lem:face diagonal}, $H_1$ cannot contain $c_{8}$ and $c_{12}$ in the same time. By Lemma \ref{lem:face diagonal}, $H_1$ cannot contain $c_{9}$ and $c_{13}$ in the same time. By Lemma \ref{lem:face diagonal}, $H_1$ cannot contain $c_{10}$ and $c_{11}$ in the same time.
Thus $H_1$ cannot contain 11 vertices of $C$, a contradiction.\\
Let us assume that $H_1$ contains $c_8$ (and does not contain the vertex $c_7$). (The cases $H_1$ contains $c_{9}$, $c_{10}$, $c_{11}$, $c_{12}$ or $c_{13}$ are similar.)
By Lemma \ref{lem:face diagonal}, $H_1$ cannot contain $c_{12}$, $c_{13}$ or $c_{16}$.
By Lemma \ref{lem:face diagonal}, $H_1$ cannot contain $c_{9}$ and $c_{11}$ in the same time. By Lemma \ref{lem:face diagonal}, $H_1$ cannot contain $c_{10}$ and $c_{14}$ in the same time.
Thus $H_1$ cannot contain 11 vertices of $C$, a contradiction.\\
Let us assume that $H_1$ contains $c_{14}$ (and does not contain the vertices $c_7,\ldots,c_{12}$ or $c_{13}$). (The cases $H_1$ contains $c_{15}$ or $c_{16}$ are similar.)
Since $H_1$ can contain the $c_1$, $c_2$, $c_3$, $c_{14},\ldots,c_{19}$,
$H_1$ cannot contain 11 vertices of $C$, a contradiction.\\
Let us assume that $H_1$ contains $c_{17}$ (and does not contain the vertices $c_7,\ldots,c_{15}$ or $c_{16}$). (The cases $H_1$ contains $c_{18}$ or $c_{19}$ are similar.)
Since $H_1$ can contain the $c_1$, $c_2$, $c_3$, $c_{17}$, $c_{18}$, $c_{19}$,
$H_1$ cannot contain 11 vertices of $C$, a contradiction.

{\bf Subcase 2.5.} The homothetic copy $H_1$ contains exactly one vertex from $S^5$.\\
Without loss of generality it may be assumed that $H_1$ contains $c_2$.
By Lemma \ref{lem:face diagonal}, $H_1$ cannot contain the vertices $c_{23}, \ldots,c_{26}$.\\
Let us assume that $H_1$ contains $c_7$. (The cases $H_1$ contains $c_{8}$, $c_{9}$ or $c_{10}$ are similar.)
By Lemma \ref{lem:face diagonal}, $H_1$ cannot contain $c_{14}$, $c_{15}$ or $c_{16}$.
By Lemma \ref{lem:face diagonal}, $H_1$ cannot contain $c_{8}$ and $c_{12}$ in the same time. By Lemma \ref{lem:face diagonal}, $H_1$ cannot contain $c_{9}$ and $c_{13}$ in the same time. By Lemma \ref{lem:face diagonal}, $H_1$ cannot contain $c_{10}$ and $c_{11}$ in the same time. By Lemma \ref{lem:face diagonal}, $H_1$ cannot contain $c_{17}$ and $c_{22}$ in the same time. By Lemma \ref{lem:face diagonal}, $H_1$ cannot contain $c_{18}$ and $c_{21}$ in the same time.
Thus $H_1$ cannot contain 11 vertices of $C$, a contradiction.\\
Let us assume that $H_1$ contains $c_{11}$ (and does not contain the vertices $c_7,\ldots,c_{9}$ or $c_{10}$). (The cases $H_1$ contains $c_{12}$, $c_{13}$, $c_{14}$, $c_{15}$, $c_{16}$ are similar.)
By Lemma \ref{lem:face diagonal}, $H_1$ cannot contain $c_{16}$.
By Lemma \ref{lem:diagonal}, $H_1$ cannot contain $c_{22}$.
By Lemma \ref{lem:diagonal}, $H_1$ cannot contain $c_{12}$ and $c_{21}$ in the same time. By Lemma \ref{lem:diagonal}, $H_1$ cannot contain $c_{13}$ and $c_{20}$ in the same time.
Thus $H_1$ cannot contain 11 vertices of $C$, a contradiction.\\
Let us assume that $H_1$ contains $c_{17}$ (and does not contain the vertices $c_7,\ldots,c_{15}$ or $c_{16}$). (The cases $H_1$ contains $c_{18},\ldots ,c_{22}$ are similar.)
Since $H_1$ can contain the $c_1$, $c_2$, $c_{17},\ldots,c_{22}$,
$H_1$ cannot contain 11 vertices of $C$, a contradiction.

{\bf Subcase 2.6.} The homothetic copy $H_1$ does not contain any vertex from $S^5$.\\
Let us assume that $H_1$ contains $c_7$. (The cases $H_1$ contains $c_{8},\ldots,c_{16}$ are similar.)
By Lemma \ref{lem:face diagonal}, $H_1$ cannot contain $c_{14}$, $c_{15}$ or $c_{16}$.
By Lemma \ref{lem:diagonal}, $H_1$ cannot contain $c_{26}$.
By Lemma \ref{lem:face diagonal}, $H_1$ cannot contain $c_{8}$ and $c_{12}$ in the same time. By Lemma \ref{lem:face diagonal}, $H_1$ cannot contain $c_{9}$ and $c_{13}$ in the same time. By Lemma \ref{lem:face diagonal}, $H_1$ cannot contain $c_{10}$ and $c_{11}$ in the same time.
By Lemma \ref{lem:face diagonal}, $H_1$ cannot contain $c_{17}$ and $c_{25}$ in the same time. By Lemma \ref{lem:face diagonal}, $H_1$ cannot contain $c_{18}$ and $c_{24}$ in the same time.
By Lemma \ref{lem:face diagonal}, $H_1$ cannot contain $c_{19}$ and $c_{20}$ in the same time.
By Lemma \ref{lem:face diagonal}, $H_1$ cannot contain $c_{21}$ and $c_{23}$ in the same time.
Thus $H_1$ cannot contain 11 vertices of $C$, a contradiction.\\
Let us assume that $H_1$ contains $c_{17}$ (and does not contain the vertices $c_7,\ldots,c_{15}$ or $c_{16}$). (The cases $H_1$ contains $c_{18},\ldots ,c_{26}$ are similar.)
By Lemma \ref{lem:face diagonal}, $H_1$ cannot contain $c_{22}$, $c_{25}$ or $c_{26}$.
Since $H_1$ can contain the $c_1$, $c_2$, $c_{17}$, $c_{18}$, $c_{19}$, $c_{20}$, $c_{21}$, $c_{23}$, $c_{24}$,
$H_1$ cannot contain 11 vertices of $C$, a contradiction.
\end{proof}

Since $\gamma_{2d+3}^d(K^d)\le\gamma_{2d+2}^d(K^d)$ and $\gamma_{2d+2}^d(K^d)={d-1\over d}$ for $d\ge 6$, $\gamma_{2d+2}^d(K^d)\le{d-1\over d}$ for $d\ge 6$.

\begin{conj}
  If $m=2d+3$, then $\gamma_m^d(K^d)={d-1\over d}$ for $d\ge 6$.
\end{conj}

\begin{thm}
  If $m=2d+4$, then $\gamma_m^d(K^d)={2d-3\over 2d-1}$ for $d=4$ and $\gamma_m^d(K^d)\le{2d-3\over 2d-1}$ for $d\ge 5$.
\end{thm}

\begin{proof}
First it will be proved that $\gamma_m^d(K^d)\le{2d-3\over 2d-1}$ for $d\ge 4$.
Let $\lambda={2d-3\over 2d-1}$ and consider the following homothetic copies of $K^d$.
$K^d_1=\lambda K^d+\left[\frac{2}{2d-1},0,0,\ldots,0\right]^T$, $K^d_2=\lambda K^d+\left[0,\frac{2}{2d-1},0,\ldots,0\right]^T$, $\ldots$, $K^d_d=\lambda K^d+\left[0,\ldots,0,\frac{2}{2d-1}\right]^T$, $K^d_{d+1}=\lambda K^d+\left[-\frac{2}{2d-1},0,0,\ldots,0\right]^T$, $\ldots$, $K^d_{2d}=\lambda K^d+\left[0,\ldots,0,-\frac{2}{2d-1}\right]^T$,
$K^d_{2d+1}=\lambda K^d+\left[{1.5\over 2d-1},{1.5\over 2d-1},0,0\ldots,0\right]^T$,
$K^d_{2d+2}=\lambda K^d+\left[-{1.5\over 2d-1},{1.5\over 2d-1},0,0\ldots,0\right]^T$,
$K^d_{2d+3}=\lambda K^d+\left[-{1.5\over 2d-1},-{1.5\over 2d-1},0,0\ldots,0\right]^T$,
$K^d_{2d+4}=\lambda K^d+\left[{1.5\over 2d-1},-{1.5\over 2d-1},0,0\ldots,0\right]^T$.
Now it will be proved that the surface of $K^d$ is covered by the homopthetic copies $K_1^d,\ldots,K^d_{2d+4}$. Consider a facet $F$ of $K^d$. Without loss of generality it can be assumed that F lies on the hyperplane $x_1+\ldots+x_d=1$. The homothetic image of $F$ with ratio $\lambda$ and centre $(1,0,0\ldots,0)$ ($(0,1,0\ldots,0),\ldots,(0,\ldots,0,1)$, resp.) is covered by $K^d_1$ ($K^d_2,\ldots,K^d_d$, resp.). Let $n_1=\left(\frac{1}{2d-1},
\frac{2}{2d-1},\frac{2}{2d-1},\ldots,\frac{2}{2d-1}\right)$,
$n_2=\left(\frac{2}{2d-1},
\frac{1}{2d-1},\frac{2}{2d-1},\ldots,\frac{2}{2d-1}\right)$,\ldots,
$n_d=\left(\frac{2}{2d-1},\ldots,\frac{2}{2d-1},\frac{1}{2d-1}\right)$.
The convex hull of the points $n_1,\ldots,n_d$ is uncovered by $K^d_1,\ldots,K^d_d$. Since 
$$\left|\left|\left({1.5\over 2d-1},{1.5\over 2d-1},0,0\ldots,0\right)-n_i\right|\right|_1$$
$$\le\max\left(2\frac{0.5}{2d-1}+(d-2)\frac{2}{2d-1},2\frac{0.5}{2d-1}+(d-3)\frac{2}{2d-1}+\frac{1}{2d-1}\right)=\frac{2d-3}{2d-1}=\lambda$$ for $i=1,\ldots,d$, the points $n_1,\ldots,n_d$ are covered by $K^d_{2d+1}$. Since $K^d_{2d+1}$ and the convex hull of the points $n_1,\ldots,n_d$ are convex bodies, $K^d_{2d+1}$ covers the convex hull of the points $n_1,\ldots,n_d$. Thus the facet $F$ is covered by $K^d_1,\ldots,K^d_d$, $K^d_{2d+1}$.
Similarly any other facet of $K^d$ is covered by $K^d_1,\ldots,K^d_{2d+4}$. Since the origin lies in each homothetic copy $K^d_1,\ldots,K^d_{2d+4}$, $K^d$ is covered by $K^d_1,\ldots,K^d_{2d+4}$ and $\gamma_{2d+4}^d(K^d)\le {2d-3\over 2d-1}$ for $d\ge 4$.

Now consider the case $d=4$. 
Let it be assumed that $\gamma_{12}^4(K^d)=\mu<{5\over 7}$.
Let $n_{1,1}=\left({1}/{7},{2}/{7},{2}/{7},{2}/{7}\right)$,
$n_{1,2}=\left({2}/{7},{1}/{7},{2}/{7},{2}/{7}\right)$,
$n_{1,3}=\left({2}/{7},{2}/{7},{1}/{7},{2}/{7}\right)$,
$n_{1,4}=(2/7,\linebreak2/7,2/7,1/7)$,
$n_{2,1}=\left(-1/7,2/7,2/7,2/7\right)$,
$\ldots$,
$n_{10,3}=\left(2/7,-2/7,1/7,-2/7\right)$,
$\ldots$,
$n_{12,3}=\left(-2/7,-2/7,-1/7,2/7\right)$,
$\ldots$,
$n_{14,1}=\left(-1/7,2/7,-2/7,-2/7\right)$
$\ldots$,
$n_{16,4}=(-2/7,-2/7,\linebreak-2/7,-1/7)$.
By Lemma \ref{lem:center is not covered}, the $8$ homothetic copies  of $K^4$ with ratio $\mu$ can cover the $8$ vertices of $K^4$ and the points $n_{1,1},\ldots,n_{16,4}$ are uncovered by these homothetic copies.
It will be proved that four homothetic copies of $K^4$ with ratio $\mu$ cannot cover the points $n_{1,1},\ldots,n_{16,4}$.
Since 
$||n_{1,1}-n_{10,3}||_1=
||n_{1,1}-n_{12,3}||_1=
||n_{1,1}-n_{14,1}||_1=
||n_{10,3}-n_{12,3}||_1=
||n_{10,3}-n_{14,1}||_1=
||n_{12,3}-n_{14,1}||_1=\frac{10}{7}=2\lambda$,
the points $n_{1,1}$, $n_{10,3}$, $n_{12,3}$ and $n_{14,1}$  cannot be covered by four homothetic copies of $K^4$ with ratio $\mu$, a contradiction.
\end{proof}

\begin{conj}
  If $m=2d+4$, then $\gamma_m^d(K^d)={2d-3\over 2d-1}$ for $d\ge 5$.
\end{conj}

%------
% Insert acknowledgments and information
% regarding funding at the end of the last
% section, i.e., right before the bibliography.
%------

%\begin{ack}
%We thank X.
%\end{ack}

%\begin{funding}
%This work was partially supported by~\ldots
%\end{funding}

%------
% Insert the bibliography.
%------

\end{document}